\documentclass[12pt]{amsart}
\usepackage{amssymb,amsmath,mathrsfs,enumerate,mathtools,comment}
\usepackage{tikz-cd}
\usepackage[pdftex,bookmarks=true]{hyperref}
\usepackage{adjustbox}
\usepackage{hhline}
\usepackage{enumerate,graphicx}
\usepackage{filecontents}
\newtheorem*{theorem*}{Theorem}
\theoremstyle{plain}

\newtheorem{theorem}{Theorem}[section]

\newtheorem{proposition}[theorem]{Proposition}
\newtheorem{lemma}[theorem]{Lemma}
\newtheorem{corollary}[theorem]{Corollary}
\theoremstyle{definition}

\newtheorem{remark}[theorem]{Remark}
\numberwithin{equation}{section}

\newcommand{\abs}[1]{\lvert#1\rvert}
\newcommand{\norm}[1]{\lVert#1\rVert}
\newcommand{\bigabs}[1]{\bigl\lvert#1\bigr\rvert}
\newcommand{\bignorm}[1]{\bigl\lVert#1\bigr\rVert}
\newcommand{\Bigabs}[1]{\Bigl\lvert#1\Bigr\rvert}
\newcommand{\Bignorm}[1]{\Bigl\lVert#1\Bigr\rVert}

\usepackage{array}

\usepackage{bbm}
\usepackage{float}

\DeclareMathOperator{\Span}{Span}

\newcommand{\N}{{\mathbb N}}
\newcommand{\E}{{\mathbb E}}
\newcommand{\bP}{{\mathbb P}}
\newcommand{\R}{{\mathbb R}}
\newcommand{\cX}{{\mathcal X}}
\newcommand{\cF}{{\mathcal F}}

\newcommand{\cS}{{\mathcal S}}
\newcommand{\cM}{{\mathcal M}}
\newcommand{\one}{\mathbf{1}}

\title{Do law-invariant linear functionals collapse to the mean?}

\author[S.~Chen]{Shengzhong Chen}
\address{Department of Mathematics, Ryerson University, 350 Victoria Street, Toronto, Canada M5B2K3}
\email{sz.chen@ryerson.ca}
\author[N.~Gao]{Niushan Gao}
\address{Department of Mathematics, Ryerson University, 350 Victoria Street, Toronto, Canada M5B2K3}
\email{niushan@ryerson.ca}

\author[D.~Leung]{Denny H.~Leung}
\address{Department of Mathematics, National University of Singapore, Singapore 117543}
\email{dennyhl@u.nus.edu}

\author[L.~Li]{Lei Li}
\address{School of Mathematical Sciences and LPMC, Nankai University, Tianjin 300071, China
}
\email{leilee@nankai.edu.cn}

\thanks{The first author is supported by an NSERC fellowship. The second author acknowledges support of an NSERC Discovery Grant.}

\keywords{Law invariance, linear functional, collapse to the mean, rearrangement-invariant space}

\subjclass[2010]{46E30,60E05}

\date{\today}

\begin{document}
\maketitle
\begin{abstract}
     In this note, we show that, on a wide range of rearrangement-invariant spaces, a law-invariant bounded linear functional is a scalar multiple of the expectation. We also construct a  rearrangement-invariant space  on which this property fails.
\end{abstract}

\section{Introduction and Notation}
In the past few years, law-invariant risk measures have been of intense research interest in Financial Mathematics; see, e.g., \cite{BC,BCb,CGX,FS08,FS12,GLMX,Joui06,Joui08,KSZ,K01,LW,Svin,TL,WR,W}. Of particular interest to us are two recent papers \cite{BCb, FM} that investigate  when a (convex) law-invariant risk measure  collapses to the mean, i.e., being  a scalar multiple  of  expectation.
We are motivated to study the problem for linear functionals. Precisely, we investigate  whether law-invariant bounded linear functionals on rearrangement-invariant spaces automatically collapse to the mean.

In Section \ref{sec2} of this paper, we collect some preliminary facts on automatically  collapsing to the mean of law-invariant bounded linear functionals. In Section \ref{sec3}, we establish  a sufficient condition  for automatically collapsing to the mean. It turns out that nearly all classical rearrangement-invariant spaces, including Lebesgue spaces $L^p(1\leq p<\infty)$, Lorentz spaces $L^{p,q}$  ($1<p<\infty, 1\leq q\leq \infty$) and Orlicz spaces,  satisfy this condition and thus law-invariant bounded linear functionals on these spaces all collapse to the mean. In Section \ref{sec4}, we construct a  rearrangement-invariant space on which some law-invariant bounded linear functionals   fail to collapse to the mean.

We introduce the terminology and notation for this paper. Throughout the paper, $(\Omega,\mathcal{F},\mathbb{P})$ stands for a non-atomic probability space. Given a random variable $X'$ on another  probability space $(\Omega',\mathcal{F}',\mathbb{P}')$, there exists a random variable  $X$ on $\Omega$ having the same distribution as $X'$, i.e, $\bP(X\leq c)=\bP'(X'\leq c)$ for any $c\in\R$; in this case, we write  $X\sim X'$ (cf., e.g., \cite[Appendix 3]{HS4}).   For a random variable $X$ on $\Omega$, define its decreasing rearrangement by $$X^*(t)=\inf\big\{\lambda>0 : \bP(\abs{X}>\lambda)\leq t\big\},\quad t\in(0,1).$$
Clearly, if $X\sim Y$, then $X^*\equiv Y^*$. Moreover, if $(0,1)$ is endowed with the Lebesgue measure, then $X^*\sim \abs{X}$.
These notions and facts  extend in a plain manner to finite measure spaces of the same measure.  In particular, we will  work on a set $A\in\mathcal{F}$ endowed with the   probability structure restricted from $(\Omega,\mathcal{F},\mathbb{P})$. Given a (measurable) partition $\pi=(A_i)_{i\in I}$ of $\Omega$, where $I$ is at most countable, we can define random variables on $\Omega$ by specifying its values on each $A_i$ and then gluing the pieces together. A useful fact is as follows. Let $\pi'=(A_i')_{i\in I}$  be a partition of $\Omega'$  for another probablity space $(\Omega',\mathcal{F}',\mathbb{P}')$, such that $\bP(A_i)=\bP'(A_i')$ for each $i\in I$. If $ {X}$ and $ X'$ are random variables on $\Omega$ and $\Omega'$, respectively, such that $ {X}\vert_{A_i}\sim  {X}'\vert_{A_i'}$ for any $i\in I$, then $X \sim  {X}'$.

Let $L^0$ be the space of all random variables (to be precise,  equivalent classes of random variables modulo a.s.\ equality) on $(\Omega,\mathcal{F},\mathbb{P})$. Throughout the paper, $\cX$ stands for a rearrangement-invariant (abbr., r.i.) space over $(\Omega,\mathcal{F},\mathbb{P})$. That is, $\cX$ is a Banach space of functions in $L^0$ such that for any $X\in \cX$, (1) if $Y\in L^0$ and $|Y|\leq |X|$ a.s.\ then $Y\in \cX$ and $\|Y\|\leq \|X\|$, and (2) if $Z\in L^0$  and $Z\sim X$ then $Z\in \cX$ and $\|Z\|=\|X\|$.
We only consider $\cX\neq\{0\}$. In this case, it is well known (see, e.g., \cite[Ch.\ 2]{BS}) that $L^\infty\subset \cX\subset L^1$ and there exist  constants $C,C'>0$ such that \begin{align}\label{normpair}
    \norm{X}_{1}\leq C\norm{X}\text{ for every }X\in\cX,\text{ and } \norm{X}\leq C'\norm{X}_{\infty}\text{ for every }X\in L^\infty .
\end{align}

Given an r.i.\ space $\cX$, its \emph{associate space} $\cX'$   is the space of all $Y\in L^0$ such that $\E[\abs{XY}]<\infty$ for every $X\in \cX$.
$\cX'$ itself is also an r.i.\ space.
In the Banach lattice language, $\cX'$ corresponds to the order continuous dual $\cX_n^\sim$ of $\cX$, which consists of all linear functionals $\rho$ on $\cX$ such that $\rho(X_n)\rightarrow \rho(X)$ whenever $(X_n)\subset \cX$ converges a.s.\ to $X\in\cX$ and there exists $X_0\in\cX$ such that $\abs{X_n}\leq X_0$ for all $n\in\N$.
The \emph{order continuous part} $\cX_a$ of $\cX$ is the set of all  functions in $\cX$ that are order continuous (in other words, have absolutely continuous norm), i.e., all $X\in \cX$ such that $\lim_{\bP(A)\rightarrow0}\norm{X\mathbf{1}_A}=0$.
It is known that when $\cX\neq L^\infty$, $\cX_a$ is the norm closure of $L^\infty$ in $\cX$. Moreover, in this case, $X\in\cX_a \iff \lim_n\norm{X\mathbf{1}_{\{\abs{X}\geq n\}}}=0$.

An r.i.\ space $\cX$ is said to be \emph{order continuous} if $\cX_a=\cX$, or equivalently, if $\cX'=\cX^*$ (\cite[Theorem 2.4.2]{MN}), where $\cX^*$ is the norm continuous dual of $\cX$.
$\cX$ is said to be \emph{monotonically complete}\footnote{In the literature, monotonic completeness bears many other names such as {\em maximal} and {\em the weak Fatou property}. It is sometimes included in the definition of function spaces.} if the pointwise supremum of any increasing, norm bounded sequence in  $\cX_+$ belongs to $\cX$, or equivalently, if $\cX$ is isomorphic to $(\cX')'$ via the evaluation mapping (\cite[Theorem 2.4.22]{MN}). In this case, there exists a constant $M>0$ such that
\begin{align}
    \label{moncomconstant}
    \norm{X}\leq M\sup_{Y\in\cX',\norm{Y}\leq 1}\bigabs{\E[XY]}\quad\text{ for all }X\in\cX,
    \end{align}
    where the norm of $Y\in\cX'$ is the dual norm on $\cX$.

 We refer the reader to \cite{AB,MN} for unexplained facts and terminology on Banach lattices and order structures and to \cite{BS} for facts and results on general r.i.\ spaces.

\section{Preliminary Observations}\label{sec2}

The functionals considered in \cite{BCb,FM} are assumed to be  $\sigma(\cX,\cX')$ lower semicontinuous. In general, weak continuities are rather restrictive assumptions to impose. For linear functionals, it is straightforward to see that under law invariance, $\sigma(\cX,\cX')$ lower semicontinuity is equivalent to collapsing to the mean.
\begin{lemma}\label{quick}
Let $\cX$ be an r.i.\ space over a non-atomic probability space $(\Omega,\cF,\bP)$. Let $\rho$ be a law-invariant  linear functional  on $\cX$.
\begin{enumerate}
    \item $\rho$ is $\sigma(\cX,\cX')$ lower semicontinuous iff $\rho(X)=\rho(\one)\E[X]$ for any $X\in\cX$.
    \item If $\rho$ is norm continuous, then $\rho(X) = \rho(\one)\E[X]$ for all $X\in \cX_a$.
\end{enumerate}
\end{lemma}

\begin{proof}
(1) Recall that a linear functional $\rho$ on $\cX$  is $\sigma(\cX,\cX')$ lower semicontinuous iff it is $\sigma(\cX,\cX')$ continuous iff it lies in $\cX'$, i.e., there exists $Y\in\cX'$ such that
\begin{align}\label{lin}
    \rho(X)=\E[XY] \text{ for any }X\in\cX.
\end{align}
Thus the ``if'' direction follows by taking $Y=\rho(\one)\one\in\cX'$. For the reverse direction, suppose that $\rho$ is given by \eqref{lin}. If $Y$ is not   constant, then there exist two sets $A,B\in\cF$ with positive probabilities and $\alpha\in \R$ such that $Y\vert_{A}>\alpha>Y\vert_{B}$. By using non-atomicity and shrinking to smaller sets, we may assume that $\bP(A)=\bP(B)$. Then $\one_A\sim \one_B$ but $\rho(\one_A)=\E[Y\one_A]>\alpha\bP(A)=\alpha\bP(B)>\E[Y\one_B]=\rho(\one_B)$, contradicting law-invariance of $\rho$. Thus $Y\equiv c$ for some constant $c\in\R$. Taking $X=\one$ in \eqref{lin}, we get $c=\rho(\one)$. Putting this back into \eqref{lin}, we get $\rho(X)=\rho(\one)\E[X]$ for any $X\in\cX$.

(2) If $\cX=L^\infty$, then $\cX_a=\{0\}$ so that there is nothing to prove. Assume that $\cX\neq L^\infty$. Then $\cX_a$ is an r.i.\ space over $\Omega$ and by order continuity of $\cX_a$, $(\cX_a)^*=(\cX_a)'$. Thus being norm continuous, $\rho\vert_{\cX_a}$ is $\sigma(\cX_a,(\cX_a)^*)=\sigma(\cX_a,(\cX_a)')$ continuous. The desired result follows from (1).
\end{proof}

Recall that for a linear functional, $\sigma(\cX,\cX')$ lower semicontinuity implies norm continuity (boundedness). Thus following Lemma \ref{quick}, it is natural to  consider what happens if $\sigma(\cX,\cX')$ lower semicontinuity is weakened to plain norm boundedness. It turns out that collapsing to the mean no longer always holds for law-invariant norm bounded linear functionals. It, however, does hold on r.i.\ spaces satisfying \eqref{convexhull} below.

Lemma \ref{reduceto+} below reduces the problem to positive functionals. For the proof, we first recall  \cite[Lemma 2.3]{CGLLa}, which will be used later in the paper as well.

\begin{lemma}[\cite{CGLLa}]\label{movingorder}
Let $X_1,X_2$ be random variables  on a non-atomic probablity space $(\Omega,\mathcal{F},\mathbb{P})$ and $X'$ be a random variable on a non-atomic probablity space $(\Omega',\mathcal{F}',\mathbb{P}')$.
 If $X'\sim X_1\geq X_2$, then there exists a random variable $X_2'$ on $\Omega'$ such that
$$X'\geq X_2'\sim X_2.$$
The conclusion still holds if both ``$\geq $'' are replaced by ``$\leq$''.
\end{lemma}

\begin{lemma}\label{reduceto+}
Let $\cX$ be an r.i.\ space over a non-atomic probability space. Let $\rho$ be a law-invariant bounded linear functional  on $\cX$. Then $\rho^\pm$ is also law invariant.
\end{lemma}

\begin{proof}
Take any $X\in\cX_+$. Recall from the Riesz-Kantorovich formula (\cite[Theorem 1.18]{AB}) that
$\rho^+(X)=\sup\{\rho(Y):0\leq Y\leq X\}$.
Take any $X',Y\in\cX$ such that $X'\sim X\geq  Y\geq 0$.
By Lemma \ref{movingorder}, there exists a random variable $Z$  such that $ X'\geq  Z\sim Y $.  Since $Y\geq 0$, $Z\geq 0$.
Thus by law invariance of $\rho$, $\rho(Y)=\rho(Z)\leq \rho^+(Z) \leq \rho^+(X')$. Taking supremum over $Y$, we have $\rho^+(X)\leq \rho^+(X')$. Hence,  $\rho^+(X) = \rho^+(X')$ by symmetry. For general random variables $X,X'\in\cX$ with $X\sim X'$, $X^\pm\sim (X')^\pm$. Thus $\rho^+(X)=\rho^+(X^+)-\rho^+(X^-)=\rho^+((X')^+)-\rho^+((X')^-)=\rho^+(X')$. Therefore, $\rho^+$ is law invariant and hence so is $\rho^-=\rho^+-\rho$.
\end{proof}

A simple, but key, observation is that for a bounded linear functional $\rho$, collapsing to the mean and being law invariant can be characterized in terms of its kernel.
Let $\cX$ be an r.i.\ space over a non-atomic probability space $(\Omega,\cF,\bP)$.
Define
$$ \cS = \Span\{X-Y: X, Y \in \cX, X\sim Y\},\text{ and }\cM = \{X\in \cX: \E[X] = 0\}.$$
Clearly, $\cM$ is a closed subspace of $\cX$ and $$\overline{\cS}\subset \cM=\{X-\E[X]\cdot\mathbf{1}:X\in\cX\},$$ where the closure is taken in $\cX$ with respect to the norm topology.  The following lemma is immediate; we skip the proof.
\begin{lemma}\label{trivial}
Let $\cX$ be an r.i.\ space over a non-atomic probability space. Let $\rho$ be a   bounded linear functional  on $\cX$.   Consider the following statements.
\begin{enumerate}
\item $\rho(X) = \rho(\one)\E[X]$ for all $X\in \cX$.
\item $\cM\subset\ker\rho$.
\item $\rho$ is law invariant.
\item $\overline{\cS}\subset\ker\rho$.
\end{enumerate}
Then (1) $\iff$ (2) $\implies$ (3) $\iff$ (4).
\end{lemma}

The following proposition reveals some elements in $\overline{\cS}$.

 \begin{proposition}
\label{conts} Let $\cX$ be an r.i.\ space over a non-atomic probability space $(\Omega,\cF,\bP)$. Then
$X- \E[X]\cdot \one \in  \overline{\cS}$ for any $X\in L^\infty$.  The same holds for any $X\in \cX_a$.
\end{proposition}

\begin{proof}
First, let $A$  be a measurable set such that $\bP(A) \leq \frac{1}{m}$ for some $m\in \N$.
By non-atomicity, there exist disjoint measurable sets $(A_i)^m_{i=1}$ of $\Omega$ such  that $A_1 = A$ and $\bP(A_i) = \bP(A)$ for all $i=1,\dots,m$.
Then
\begin{align}\label{smallcom}
    \one_A - \bP(A)\cdot \one = \frac{1}{m}\sum^m_{i=1}(\one_A - \one_{A_i}) +\frac{1}{m}\one_{\bigcup_{i=1}^mA_i}-\bP(A)\cdot\one.
\end{align}
In particular, if $\bP(A)=\frac{1}{m}$, then $\bP(\bigcup_{i=1}^mA_i)=1$, so that $\one_{\bigcup_{i=1}^mA_i}=\one$. Hence $\one_A-\bP(A)\cdot\one=\frac{1}{m}\sum^m_{i=1}(\one_A - \one_{A_i}) \in\cS$.

Next, let $A$ be a measurable set such that $\bP(A) = \frac{n}{m}$ for some $m,n \in \N$ with $1\leq n < m$.
Partition $A$ into measurable sets $B_1,\dots, B_n$, where $\bP(B_i) = \frac{1}{m}$ for   $i=1,\dots,n$.
By the above, $\bP(B_i)-\frac{1}{m}\cdot \one \in \cS$ for all $i=1,\dots,n$.
Hence
\[ \one_A - \frac{n}{m}\cdot \one = \sum^n_{i=1}\Big(\one_{B_i}-\frac{1}{m}\cdot \one\Big) \in \cS.\]

Now, let $A$ be any measurable set  with $\bP(A)\in(0,1)$. For any $m\geq 3$,  there exist disjoint subsets $B,C$ of $A$ such that $A=B\cup C$, $\bP(B)=\frac{n-1}{m}$ and $0\leq \bP(C)<\frac{1}{m}$, where $n\in[1,m]\cap\N$. Clearly, $\one_A-\bP(A)\cdot\one=\one_B-\bP(B)\cdot\one+\one_C-\bP(C)\cdot\one$. By the second case above, $\one_B-\bP(B)\cdot\one\in\cS$. By\eqref{smallcom}, $\mathrm{d}(\one_C-\bP(C)\cdot\one,\cS)\leq\frac{2}{m}\norm{\mathbf{1}}$.  Therefore, $$\mathrm{d}(\one_A-\bP(A)\cdot\one,\cS)\leq \frac{2}{m}\norm{\one}.$$
Since $m\geq 3$ is arbitrary,  we obtain   $\one_A-\bP(A)\cdot\one\in\overline{\cS}$.

From the above, it follows immediately  that $X - \E[X]\cdot \one \in \overline{\cS}$ for all simple functions $X$.
If $X \in L^\infty$, take a sequence of simple functions $(X_n)$ such that $\|X_n-X\|_\infty \to 0$.
Since $X_n - \E[X_n]\cdot \one \in \overline{\cS}$ for all $n\in\N$ and the sequence converges to $X - \E[X]\cdot \one$ in the norm of $\cX$ by \eqref{normpair},
$X - \E[X]\cdot \one \in \overline{\cS}$. This proves the first assertion.

The second assertion follows from the fact that, if $\cX_a\neq\{0\}$,  then $\cX_a$ is the norm closure of $L^\infty$ in $\cX$.
\end{proof}

\begin{remark}\label{remarkinfty}
  \begin{enumerate}
   \item Taking $\cX=L^\infty$ in   Proposition \ref{conts} and applying the implication (3)$\implies$(4) in Lemma \ref{trivial}, we obtain that if $\rho$ is a  law-invariant  bounded linear functional on $L^\infty$, then it collapses to the mean. This fact, which may be known in the literature,   can also be proved using the Radon-Nikydom theorem,  by setting $\mu(A)=\rho(\one_A)$ for $A\in\cF$. We skip the details.
      \item Similarly, the assertion for $\cX_a$ in Proposition \ref{conts}  implies Lemma \ref{quick}(2). The proof in Lemma \ref{quick}(2) also uses the Radon-Nikydom theorem, although implicitly,  when realizing $(\cX_a)^*$ as $(\cX_a)'$. We point out that Lemma \ref{quick}(2) implies that law-invariant bounded linear functionals on order continuous r.i.\ spaces   collapse to the mean, which is also likely known in the literature.
  \end{enumerate}
\end{remark}

 Below is an interesting consequence of Proposition   \ref{conts}/Remark \ref{remarkinfty}(1).
 \begin{corollary}\label{lowercontrol}
Let $\cX$ be an r.i.\ space over a non-atomic probability space and $\rho$ be a  law-invariant positive  linear functional on $\cX$. Then $\rho(X)\geq \rho(\one)\E[X]$ for any $X\in\cX_+$.
 \end{corollary}

\begin{proof}
Let $X\in\cX_+$. By Remark \ref{remarkinfty}(1), $\rho\vert_{L^\infty}$ collapses to the mean. Thus for any $n\geq 1$, since $X\one_{\{X\leq n\}}\in L^\infty$, $\rho(X)\geq \rho(X\one_{\{X\leq n\}})=\rho(\one)\E[X\one_{\{X\leq n\}}]$. Letting $n\rightarrow \infty$, we get $\rho(X)\geq \rho(\one)\E[X]$.
\end{proof}

We now present some equivalent conditions for collapsing to the mean.

\begin{proposition}\label{equicond}
Let $\cX$ be an r.i.\ space over a non-atomic probability space. The following are equivalent.
\begin{enumerate}
\item Every  law-invariant  positive linear functional $\rho$ on $\cX$ has the form  $\rho(X) = \rho(\one)\E[X]$ for all $X\in \cX$.
\item Every law-invariant  bounded linear functional $\rho$ on $\cX$ has the form  $\rho(X) = \rho(\one)\E[X]$ for all $X\in \cX$.
\item $\overline{\cS}  = \cM$.
\item $\cM \subseteq \overline{\cS - \cX_+}$.
\end{enumerate}
\end{proposition}

\begin{proof}
As   remarked earlier, (1)$\iff$(2) follows from Lemma \ref{reduceto+}.
If (2) holds, then by Lemma \ref{trivial}, any bounded linear functional on $\cX$ that vanishes on $\overline{\cS}$ also vanishes on  $\cM$.  Hence, $\cM \subseteq \overline{\cS}$.  The reverse inclusion is trivial.  Thus (3) holds.
 (3)$\implies$(4) is clear.
Assume that (4) holds.  Then
\[ \cX = \R\cdot \one + \cM \subseteq \R\cdot \one +\overline{\cS - \cX_+}.\]
Hence, $\R\cdot \one + \cS -\cX_+$ is dense in $\cX$.
Let $\rho$ be a law-invariant  positive functional on $\cX$.
Then the bounded linear functional $\tau(X) := \rho(X) - \rho(\one)\E[X]$   vanishes on $\R\cdot\one$ and $\cS$ and $\tau\leq 0$ on $-\cX_+$ by Corollary \ref{lowercontrol}. It follows that
$\tau(X) \leq 0$ for all $X \in \R\cdot \one + \cS -\cX_+$.
 Hence $\tau(X) \leq 0$ for all $X\in \cX$.  Thus $\tau = 0$, i.e., (1) holds.
\end{proof}

\section{A Sufficient Condition}\label{sec3}

The following property was shown by the authors in \cite{CGLLa} to play a pivotal role in  automatic continuity of law-invariant convex functionals:
\begin{align}\label{convexhull}
    \mathrm{d}\big(\mathcal{CL}(X), \cX_a\big) =0\text{ for any }X\in\cX_+, \text{ where } \mathcal{CL}(X)=\mathrm{co}(\{Y: Y\sim X\}).
\end{align}
It was termed there as the \emph{Almost Order Continuous Equidistributional Average (AOCEA)} property.
Here, we show that the AOCEA property is also sufficient to ensure that all law-invariant bounded linear functionals collapse to the mean.

\begin{proposition}\label{sufficient}
Let $\cX$ be an r.i.\ space over a non-atomic probability space other than $L^\infty$ that satisfies the AOCEA property.
Then $\overline{\cS} = \cM$.  Hence every   law-invariant bounded linear functional $\rho$ on $\cX$ has the form  $\rho(X) = \rho(\one)\E[X]$ for all $X\in \cX$.
\end{proposition}

\begin{proof}
By \eqref{normpair},  there is a finite constant $C$ such that $\norm{X}_1 \leq C\norm{X}$ for all $X \in \cX$.
We may also assume that $\norm{\mathbf{1}} =1$. Let $X\in \cX_+$ and let $\varepsilon > 0$ be given.
By \eqref{convexhull}, choose $X' \in \mathcal{CL}(X)$ and $V\in\cX_a$ such that $\norm{X'-V}<\varepsilon$.
By Proposition \ref{conts}, $$V-\E[V]\cdot\one\in\overline{\cS}.$$
From the equation $X'-\E[X']\cdot \one=(X'-V)-\E[X'-V]\cdot\one+V-\E[V]\cdot \one$,  one sees that
$$\mathrm{d}(X'-\E[X']\cdot\one,\overline{\cS}) \leq \norm{X'-V}+\norm{X'-V}_1\norm{\one}<(C+1)\varepsilon.$$
Since $X-X'\in \cS$ and $\E[X]=\E[X']$, it follows from $X-\E[X]\cdot\one=(X-X')+X'-\E[X']\cdot \one$ that
\[ d(X-\E[X]\cdot\one, \overline{\cS}) \leq (C+1) \varepsilon.\]
By arbitrariness of $\varepsilon$, $X-\E[X]\cdot \one \in \overline{\cS}$ for all $X \in \cX_+$.
It follows easily that $X- \E[X]\cdot \one \in \overline{\cS}$ for all $X\in \cX$. Therefore, $\cM \subseteq \overline{\cS}$.
The reverse inclusion is clear.  This proves that $\cM = \overline{\cS}$. The second asssertion follows from Proposition \ref{equicond}.
\end{proof}

An alternative proof of this proposition  in the spirit of \cite{CGLLa} is included in the appendix.

It  was also proved in \cite{CGLLa} that when $\cX\neq L^\infty$, \eqref{convexhull} is equivalent to the following  condition:
\begin{align}
\nonumber&\text{For any }X \in \cX_+,  \text{any } \varepsilon >0,  \text{ and any sequence of measurable sets } (A_n)_{n=1}^\infty \\
&\label{3.2}
\text{with }A_n \downarrow \emptyset, \text{ there exist natural numbers }(n_i)_{i=1}^k, \text{  random variables } (Z_i)_{i=1}^k,  \\
\nonumber&\text{and  a convex combination }\sum\nolimits_{i=1}^k\lambda_i Z_i \text{ such that }Z_i \sim X\mathbf{1}_{A_{n_i}}  \text{  for } i=1,\dots,k \\
\nonumber&\text{and }\Bignorm{\sum\nolimits_{i=1}^k\lambda_i Z_i} <\varepsilon.
\end{align}
We   compare it with the following more verifiable condition that uses disjoint $Z_i$'s:
\begin{align}
\nonumber&\text{For any } X\in \cX, \text{ any }\varepsilon>0,  \text{ any   sequence of measurable sets }   (A_n)_{n=1}^\infty  \\
&\label{3.3} \text{with } A_n \downarrow \emptyset \text{ and }\sum_{n=1}^\infty\bP(A_n)\leq 1,
\text{ and any disjoint sequence of measurable}  \\\nonumber & \text{functions }(Z_n)_{n=1}^\infty  \text{ such}
\text{  that } Z_n\sim X\one_{A_n}  \text{ for all } n\geq 1, \text{ there exist }n_1,\dots,
\\ \nonumber&n_k\in \N \text{ such that } \Bignorm{\frac{1}{k}\sum^k_{i=1}Z_{n_i}} <\varepsilon.
\end{align}

\begin{proposition}\label{twocondi}
Let $\cX$ be an r.i.\ space over a non-atomic probability space $(\Omega,\cF,\bP)$ other than $L^\infty$. Then \eqref{3.3}$\implies$\eqref{3.2} (equivalently, \eqref{convexhull}). The reverse is true if $\cX$ is  either monotonically complete or order continuous.
\end{proposition}

\begin{proof}[Proof of Proposition \ref{twocondi}; the first assertion] Assume that \eqref{3.3} holds. To verify  \eqref{3.2}, take any $X\in\cX_+$, $A_n\downarrow\emptyset$, and $\varepsilon>0$. By passing to a subsequence of $(A_n)$, we may assume that $\sum_{n=1}^\infty\bP(A_n)\leq 1$. By non-atomicity, there exists a disjoint sequence of measurable sets $(B_n)_{n=1}^\infty$ such that $\bP(A_n)=\bP(B_n)$ for all $n\geq 1$. For each $n\in\N$, since $\bP(X\one_{A_n}\neq 0)\leq \bP(A_n)=\bP(B_n)$, by non-atomicity again, there exists a random variable $Z_n$, supported in $B_n$, such that $Z_n\sim X\one_{A_n}$. In particular, $Z_n$'s are disjoint. Applying \eqref{3.3} to $(Z_n)$, we obtain the desired $Z_i$'s and convex combination.\end{proof}

For the proof of \eqref{3.2}$\implies$\eqref{3.3} in the second assertion, we need
some technical lemmas.
For random variables  $X_1,\dots,X_k$ such that $\sum^k_{i=1}\bP(X_i\neq 0)\leq 1$, let $$\bigoplus^k_{i=1}X_i = \sum^k_{i=1}Y_i,$$
where $Y_i$'s are disjoint and $Y_i\sim X_i$ for each $i$. Clearly, $$\Bigabs{\bigoplus^k_{i=1}X_i}=\Bigabs{\sum^k_{i=1}Y_i}=\sum^k_{i=1}\abs{Y_i}=\bigoplus^k_{i=1}\abs{X_i}.$$
Although $\bigoplus^k_{i=1}X_i$ is not uniquely defined as a random variable, its distribution and hence its norm are uniquely determined by $X_i$'s and that are what we use.

\begin{lemma}\label{another}
Let $\cX$ be an r.i.\ space over a non-atomic probability space $(\Omega,\cF,\bP)$ and $X\in\cX$.  Let $A\in\cF$ be such that $\bP(A)\leq \frac{1}{k}$ for some $k\in\N$  and  that
\begin{align}\label{cccc}
\alpha:=\inf\{\abs{X(\omega)}: \omega \in A\} \geq \|X\one_{A^c}\|_\infty.
\end{align}
Let $A_i\in\cF$ be such that $\bP(A_i)\leq \bP(A)$, $i=1,\dots,k$. Then $\bignorm{\bigoplus_{i=1}^k X\one_{A_i}}\leq \bignorm{\bigoplus_{i=1}^kX\one_{A}}$.
\end{lemma}

\begin{proof}
By replacing $X$ with $\abs{X}$, we  assume that $X\geq 0$.
For any $i=1,\dots,k$, since $\bP(A_i\cap A^c)=\bP(A_i)-\bP(A_i\cap A)\leq \bP(A)-\bP(A_i\cap A)=\bP(A\cap A_i^c)$, take $B_i\subset A\cap A_i^c$ such that $\bP(B_i)=\bP(A_i\cap A^c)$. By non-atomicity, find $Z_i\vert_{B_i}\sim X\vert_{A_i\cap A^c}$. Then by \eqref{cccc}, $Z_i\vert_{B_i}\leq \alpha\leq X\vert_{B_i}$. Furthermore, set $Z_i =X$ on ${A_i\cap A}$ and $Z_i=0$ off $B_i\cup (A_i\cap A)$.
Then   $Z_i\sim X\one_{A_{i}}$ and  $Z_i\leq {X}\one_{A}$. Since $\bP(A)\leq \frac{1}{k}$, we can take disjoint random variables $U_i$'s such that $U_i\sim X\one_A$, $i=1,\dots,k$. By Lemma \ref{movingorder} applied to the chain $U_i\sim X\mathbf{1}_A\geq Z_i $, there exist random variables $V_i$'s such that $U_i\geq V_i\sim Z_i$, $i=1,\dots,k$. In particular, $V_i$'s are disjoint. Then $0\leq \bigoplus_{i=1}^k X\one_{A_i}=\sum_{i=1}^k V_i\leq \sum_{i=1}^k U_i=\bigoplus_{i=1}^k X\one_A$ so that $\bignorm{\bigoplus_{i=1}^k X\one_{A_i}}\leq \bignorm{\bigoplus_{i=1}^kX\one_{A}}$.
\end{proof}

For two random variables $X, Y$, write
\[Y
\prec X \quad\text{ if  } \int^s_0 Y^*\,\mathrm{d}t \leq \int^s_0 X^*\,\mathrm{d}t\text{ for all } s\in [0,1].\]
Slightly modifying the proof of \cite[Ch.\ 2, Corollary 4.7]{BS}, one obtains that if $\cX$ is monotonically complete and two random variables $X,Y$ satisfy that $X\in\cX$ and $Y\prec X$, then $Y\in \cX$ and $\norm{Y}\leq M\norm{X}$, where $M$ is the constant in \eqref{moncomconstant}.

\begin{lemma}\label{marjorization}
Let $X,X_1,\dots,X_k$ be non-negative random variables over a non-atomic probability space $(\Omega,\cF,\bP)$ such that $\bP(X\neq 0)\leq \frac{1}{k}$ and $X_i\sim X$ for $i=1,\dots,k$. Let $\lambda_i\geq 0$, $i=1,\dots,k$, be such that $\sum^k_{i=1}\lambda_i =1$. Then
\[ \frac{1}{k} \bigoplus^k_{i=1}X  \prec \sum^k_{i=1}\lambda_i X_i.\]
If in addition $X$ lies in a monotonically complete  r.i.\ space $\cX$ over  $\Omega$, then $\norm{\frac{1}{k} \bigoplus^k_{i=1}X}\leq M\norm{\sum^k_{i=1}\lambda_i X_i}$, where $M$ is the constant in \eqref{moncomconstant} for $\cX$.
\end{lemma}

\begin{proof}For any $s\in [0,1]$, by \cite[Ch.\ 2, Lemma 2.5]{BS}, there exists a measurable set $A_i$ with $\bP(A_i) = \frac{s}{k}$  such that
\[\int_{A_i}X_i = \int^{\frac{s}{k}}_0X^*_i\,dt.\]
Let $A = \bigcup^k_{i=1}A_i$.  Then $\bP(A) \leq s$.
Hence by \cite[Ch.\ 2, Lemma 2.1]{BS},
\begin{align*}
 \int^s_0\Big(\sum^k_{i=1}\lambda_i X_i\Big)^*\,\mathrm{d}t &\geq \int_A\sum^k_{i=1}\lambda_iX_i \geq \sum^k_{i=1}\lambda_i\int_{A_i}X_i \\&= \sum^k_{i=1}\lambda_i\int^{\frac{s}{k}}_0X^*_i\,\mathrm{d}t = \int^{\frac{s}{k}}_0X^*\,\mathrm{d}t\\
 &= \int^s_0 \Big(\frac{1}{k}\bigoplus^k_{i=1}X\Big)^*\,\mathrm{d}t,
 \end{align*}
 where the last equality  follows from the fact $\big(\frac{1}{k}\bigoplus^k_{i=1}X\big)^*(t)=\frac{1}{k}X^*(\frac{t}{k})$. This proves the first assertion.

The additional assertion follows from the remark preceding the lemma.
\end{proof}

\begin{proof}[Proof of Proposition \ref{twocondi}; the second assertion]
Assume that \eqref{3.2} holds. Let $X\in \cX$,  $\varepsilon > 0$, and $A_n\downarrow\emptyset$ with $\sum_{n=1}^\infty \bP(A_n)\leq 1$ be given.  Note that \eqref{3.3} is equivalent to finding $n_1,\dots, n_k$ such that $\norm{\frac{1}{k}\bigoplus^k_{i=1}X\one_{A_{n_i}}} < \varepsilon$.
If $X\in\cX_a$ (in particular, if $\cX$ is order continuous), then $\lim_{\bP(A)\rightarrow0}\norm{X\one_A}=0$. Thus  it is easy to find the desired $n_i$'s. For the following, we assume that $X\not\in \cX_a$ and $\cX$ is monotonically complete.

Take a strictly increasing sequence of positive numbers $(\alpha_m)_{m=1}^\infty$ such that $\sum_{m=1}^\infty\bP(\abs{X}\geq \alpha_m)\leq 1$.
By  \eqref{3.2} applied to $\abs{X}$ and the sequence of measurable sets $(\{\abs{X}\geq \alpha_m\})_{m=1}^\infty$, we obtain  natural numbers $m_1<\cdots< m_k $, random variables $(Z_i)_{i=1}^k$ and a convex combination $\sum_{i=1}^k\lambda_iZ_i$ such that $$Z_i\sim \abs{X}\mathbf{1}_{\{\abs{X}\geq \alpha_{m_i}\}}\text{ for } i=1,\dots,k\quad\text{ and }\quad\Bignorm{\sum_{i=1}^k\lambda_iZ_i}<\frac{\varepsilon}{M},$$
where $M$ is the constant in \eqref{moncomconstant} for $\cX$.
For $i=1,\dots,k$, since $Z_i\sim \abs{X}\mathbf{1}_{\{\abs{X}\geq \alpha_{m_i}\}}\geq \abs{X}\mathbf{1}_{\{\abs{X}\geq \alpha_{m_k}\}}$, by Lemma \ref{movingorder}, there exists  a random variable $Z_i'$ such that $Z_i\geq Z_i'$ and $$Z_i'\sim \abs{X}\mathbf{1}_{\{\abs{X}\geq \alpha_{m_k}\}}.$$
Since $0\leq Z_i'\leq Z_i$ for each $i$, $$\Bignorm{\sum_{i=1}^k\lambda_iZ_i'}<\frac{\varepsilon}{M}.$$
Clearly, $k\bP(\abs{X}\geq \alpha_{m_k})\leq \sum_{i=1}^k\bP(\abs{X}\geq \alpha_{m_i})\leq 1$ so that $\bP(\abs{X}\geq \alpha_{m_k})\leq\frac{1}{k}$.
By Lemma \ref{marjorization} applied to $\abs{X}\mathbf{1}_{\{\abs{X}\geq \alpha_{m_k}\}}$ and $Z_i'$'s, $$\Bignorm{\frac{1}{k} \bigoplus^k_{i=1}\abs{X}\mathbf{1}_{\{\abs{X}\geq \alpha_{m_k}\}} }<\varepsilon.$$
Since $X\notin \cX_a$, $X\notin L^\infty$, consequently $\bP(\abs{X}\geq \alpha_{m_k})>0$.
Thus we can take $n_1,\dots,n_k\in\N$ such that $\bP(A_{n_i})\leq \bP(\abs{X}\geq \alpha_{m_k})$ for $i=1,\dots,k$.
By Lemma \ref{another} applied to $A=\{\abs{X}\geq \alpha_{m_k}\}$, we obtain
$\norm{\frac{1}{k}\bigoplus^k_{i=1}X\one_{A_{n_i}}} < \varepsilon$, as desired.
\end{proof}

\begin{remark}
It was shown in \cite{CGLLa} that the AOCEA property \eqref{convexhull} is    satisfied by nearly all classical r.i.\ spaces other than $L^\infty$, including  Lebesgue spaces $L^p(1\leq p<\infty)$, Lorentz spaces $L^{p,q}$  ($1<p<\infty, 1\leq q\leq \infty$) and Orlicz spaces.
Proposition \ref{twocondi} gives another class of r.i.\ spaces with the AOCEA property: Let $\cX$ be an r.i.\ space over a non-atomic probability space other than $L^\infty$. If $\cX$ satisfies an upper $p$-estimate for some $p\in(1,\infty)$ (see, e.g., \cite[Definition 1.f.4]{LT}), then it clearly satisfies \eqref{3.3} and hence \eqref{3.2} and \eqref{convexhull}. By Proposition \ref{sufficient}, law-invariant bounded linear functionals on all of these spaces collapse to the mean.\end{remark}

The AOCEA property \eqref{convexhull}, however, is not universally satisfied by all r.i.\ spaces. In the next section, we construct an r.i.\ space on which there is a positive law-invariant linear functional that does not collaspe to the mean and thus, in particular, \eqref{convexhull} fails in this space.

\section{A Counterexample}\label{sec4}
For simplicity, we work on $[0,1]$ endowed with the Lebesgue measure. Let $\cX$ be the space of all random variables $X$ on $[0,1]$ such that
\[ \norm{X}:=   \sup_{n\in\N}\Big(n2^n\int^{\frac{1}{2^n\cdot n!}}_0X^*\mathrm{d}t \Big)< \infty.\]
Then $(\cX,\norm{\cdot})$ is an r.i.\ space over $[0,1]$.
Indeed, law-invariance of $\cX$ and $\norm{\cdot}$ is obvious.  Other defining properties of an r.i.\ space can be easily verified using results in \cite[Ch.\ 2]{BS}, in particular, Proposition 1.7 there.   In fact, $\norm{\cdot}$ is obviously positive homogeneous.  To verify   the  triangle inequality, recall from \cite[P.\ 54]{BS} that for any $s\in[0,1]$ and any random variables $X,Y$, $$\int_0^s(X+Y)^*\mathrm{d}t\leq \int_0^sX^*\mathrm{d}t+\int_0^sY^*\mathrm{d}t.$$
It is then immediate that if $X,Y\in\cX$ then $X+Y\in\cX$ and $\norm{X+Y}\leq \norm{X}+\norm{Y}$.
Next, we check that $\cX$ is complete. Since $X^*$ is decreasing and $X^*\sim \abs{X}$, $$\norm{X}\geq 2\int_0^{\frac{1}{2}}X^*\mathrm{d}t\geq \int_0^1X^*\mathrm{d}t=\norm{X}_1.$$
Thus  if $X_n$'s   are random variables such that  $\sup_{n\in\N}\norm{X_n}<\infty$ and $0\leq X_n\uparrow$ a.s., then $\sup_{n\in\N}\norm{X_n}_1<\infty$, so that there exists a random variable $X\in L^1$ such that $X_n\uparrow X$ a.s. By \cite[Ch.\ 2, Proposition 1.7]{BS}, $X_n^*(t)\uparrow X^*(t)$ for any $t\in(0,1)$, so that $\int^{s}_0X_n^*\mathrm{d}t\uparrow \int^{s}_0X^*\mathrm{d}t$ for any $s\in[0,1]$. It follows that $\norm{X_n}\uparrow \norm{X}$. To sum up, we have \begin{align}\label{examplem}
  \text{if } \sup_{n\in\N}\norm{X_n}<\infty \text{ and } 0\leq X_n\uparrow, \text{ then }X:=\lim_nX_n\in \cX \text{ and }\norm{X_n}\uparrow \norm{X}.
\end{align}
This in turn implies the Riesz-Fisher property, i.e, if $(X_n)$ is a sequence in $\cX$ such that $\sum_{n=1}^\infty \norm{X_n}<\infty$, then $\sum_{n=1}^\infty X_n$ converges in $\cX$ in norm (also almost surely). It is well known that the Riesz-Fisher property implies completeness of the norm. Furthermore, \eqref{examplem} actually says that  $(\cX,\norm{\cdot})$ is monotonically complete.

Following the notation in Section \ref{sec2}, set $\cS = \Span\{U-V: U, V\in \cX, U\sim V\}$ and $\cM = \{X\in \cX: \E[X] =0\}$.  We  show that $\cS$ is not dense in $\cM$. This, together with Proposition \ref{equicond}, implies that some  law-invariant  positive linear functional on $\cX$ does not collapse to the mean. Consequently,  $\cX$ fails \eqref{convexhull} and \eqref{3.2} by Proposition \ref{sufficient}  and also fails   \eqref{3.3} by Proposition \ref{twocondi}.

We begin with the following lemma.

\begin{lemma}
Let $m,n\in\N$ and $\varepsilon >0$ be given.
Suppose that $Y,Z, U_i,V_i$ ($1\leq i\leq m$) are random variables in $\cX_+$ such that $\|Z\| < \varepsilon$, $U_i\sim V_i$, $i=1,\dots,m$,  and that \[U: =\sum^m_{i=1}U_i \geq Y- Z +\sum^m_{i=1}V_i.\]
Then for any measurable set $A$ with $\bP(A) \leq \frac{1}{(m+1)2^nn!}$, there exists a measurable set $A'$ containing $A$ such that $\bP(A') \leq (m+1)\bP(A)$ and that
\[ \E[U\one_{A'}] - \E[U\one_A]\geq \E[Y\one_{A}]  -\frac{\varepsilon}{n 2^n}.\]
\end{lemma}

\begin{proof}
For $1\leq i\leq m$,  by Lemma \ref{movingorder}, there exists   a random variable $W_i$   such that    $U_i\one_A \sim W_i\leq V_i$. Set $A'=A\cup \bigcup^m_{i=1}\{W_i\neq0\}$. Then $\bP(A')\leq (m+1)\bP(A)$. Since $V_i\one_{A'}\geq W_i\one_{A'}=W_i$, we have
\begin{align*}
\E[U\one_{A'}] &\geq \E[Y\one_{A'}] + \sum^m_{i=1}\E[V_i\one_{A'}] - \E[Z\one_{A'}]\\
&\geq \E[Y\one_{A'}] + \sum^m_{i=1}\E[W_i] - \E[Z\one_{A'}]\\
&= \E[Y\one_{A'}] + \sum^m_{i=1}\E[U_i\one_{A}] - \E[Z\one_{A'}]\\
& \geq  \E[Y\one_{A}] +\E[U\one_{A}] - \E[Z\one_{A'}].
\end{align*}
Since $\bP(A') \leq (m+1)\bP(A) \leq \frac{1}{2^nn!}$, we have, by definition of $\norm{\cdot}$,
\[ \E[Z\one_{A'}] \leq  \int_0^{\frac{1}{2^nn!}}Z^*\mathrm{d}t\leq  \frac{1}{n2^n}\cdot\norm{Z} \leq \frac{\varepsilon}{n2^n},
\]
where the first inequality is due to \cite[Ch.\ 2, Lemma 2.1]{BS}.
This proves the lemma.
\end{proof}

Iterating the lemma $k$ times gives the following lemma.

\begin{lemma}\label{4.2}
Let $m,n, k \in\N$ and   $\varepsilon >0$ be given.
Suppose that $Y,Z, U_i,V_i$ ($1\leq i\leq m$) are random variables in $\cX_+$ such that $\|Z\| < \varepsilon$, $U_i\sim V_i$, $i=1,\dots,m$, and that \[U: =\sum^m_{i=1}U_i \geq Y- Z +\sum^m_{i=1}V_i.\]
For any measurable set $A$ with $\bP(A) \leq \frac{1}{(m+1)^k2^nn!}$, there exists a measurable set $A'$ containing $A$ such that $\bP(A') \leq (m+1)^k\bP(A)$ and that
\[ \E[U\one_{A'}] \geq k\E[Y\one_{A}] + \E[U\one_A] -\frac{k\varepsilon}{n 2^n}.\]
\end{lemma}

We are ready to complete the proof of that $\overline{\cS}\subsetneq \cM$.

\begin{proof}[Proof of $\overline{\cS}\subsetneq \cM$]
Set $$c_n = \frac{1}{2^n\cdot(n+1)!}\quad\text{ for any }n\in\N,$$
and $$Y = \sum^\infty_{n=3}n!\one_{[c_{n+1},c_n)}.$$
We claim that $Y\in\cX$. Indeed, since $Y $ is non-negative and decreasing on $(0,1)$, $Y^*=Y$ a.s. For any $l\geq 4$, since $c_l<\frac{1}{2^l\cdot l!}<c_{l-1}$,
\begin{align*}
\int^{\frac{1}{2^l\cdot l!}}_0Y^*\mathrm{d}t  \leq  \frac{(l-1)!}{2^l\cdot l!}+\sum^\infty_{n=l}n!c_n
= \frac{1}{l2^l}+ \sum^\infty_{n=l}\frac{1}{2^n(n+1)} \leq \frac{3}{l2^l}.
\end{align*}
It follows that $Y\in \cX$, as claimed. Clearly,   $X: = Y - \E[Y]\cdot\one \in \cM$.

We claim that $X\notin \overline{\cS}$.
Assume on the contrary that $X\in \overline{\cS}$. By splitting a random variable into the positive and negative parts, it is easy to see that  $\cS = \Span\{U-V: U, V\in \cX, U\sim V, U,V\geq 0\}$.
Thus there are sequences $(U_i)^m_{i=1}, (V_i)^m_{i=1}$ in $\cX_+$ and $W\in \cX$ with $\|W\| <\frac{1}{8}$ such that $U_i\sim V_i$ for $i=1,\dots,m$ and $X = \sum^m_{i=1}(U_i-V_i) +W$.
Hence
\[ U:= \sum^m_{i=1}U_i \geq Y - Z + \sum^m_{i=1}V_i,\]
where  $Z = \E[Y]\cdot \one + W^+$.
Observe that $\E[Y]=\sum^\infty_{n=3}n!(c_n-c_{n+1})\leq \sum_{n=3}^\infty \frac{1}{2^n(n+1)} <\frac{1}{8}$ and that $\norm{\one}=1$. Hence
 $$\norm{Z} \leq \norm{\E[Y]\cdot \one} + \norm{W} < \frac{1}{4}.$$

For any $k\in\N$, let $n\geq 6$ be such that
\[ c_n - c_{n+1} = \frac{n+\frac{3}{2}}{2^n\cdot(n+2)!} \leq \frac{1}{ (m+1)^k2^n\cdot n!}.\]
Taking $A = (c_{n+1},c_n)$ and $\varepsilon=\frac{1}{4}$ and applying Lemma \ref{4.2}, we obtain a measurable set $A'$ such that $\bP(A') \leq (m+1)^k\bP(A)\leq \frac{1}{2^n\cdot n!}$
and that
\begin{align*} \E[U\one_{A'}] \geq k\E[Y\one_A] - \frac{k }{4n2^n}  = \frac{k}{2^n}\Big(\frac{n+\frac{3}{2}}{(n+1)(n+2)} - \frac{1}{4n}\Big)> \frac{k}{2^{n+2}n}.
\end{align*}
Therefore,
\[ \|U\| \geq n2^n\int^{\frac{1}{2^n\cdot n!}}_0U^*\mathrm{d}t \geq n2^n\E[U\one_{A'}] \geq \frac{k}{4}.
\]
This cannot hold for all $k\in\N$ and concludes the proof.
\end{proof}

\appendix
\section{An alternative proof of Proposition \ref{sufficient}}

\begin{proof}[Alternative proof]
Let $\rho$ be a positive law-invariant linear functional on $\cX$.
Recall from Lemma \ref{lin} that $\rho(X) = \rho(\one)\E(X)$ for all $X\in \cX_a$.
Let $X\in \cX_+$.  By \cite[Theorem 2.2]{CGLLa}, $-\rho$ has the Fatou property
at $0$, i.e., $-\rho(X)\leq \liminf_n\big(-\rho(X_n)\big)$ whenever $X_n\stackrel{a.s.}{\rightarrow}0$ and there exists $X_0\in\cX$ such that $\abs{X_n}\leq X_0$ for all $n\geq 1$. In particular, take $X_n=X\one_{\{X\geq n\}}$ for $n\geq 1$. Since the sequence $(-\rho(X\one_{\{X \geq n\}}))$ is increasing,
\[ 0 = -\rho(0) \leq \liminf_n -\rho(X\one_{\{X \geq n\}}) = -\lim_n \rho(X\one_{\{X \geq n\}})\leq 0.
\]
 Thus $\lim_n \rho(X\one_{\{X \geq n\}}) =0$.
Therefore,
$ \rho(X) = \lim_n [\rho(X\one_{\{X < n\}}) + \rho(X\one_{\{X \geq n\}})] = \lim_n \rho(X\one_{\{X < n\}})$. Since $X\one_{\{X < n\}} \in L^\infty\subset \cX_a$,
\[ \rho(X) = \lim_n \rho(X\one_{\{|X| < n\}}) = \lim_n \rho(\one)\E(X\one_{\{|X| < n\}}) = \rho(\one)\E(X).\]
Since $\rho(X) = \rho(\one)\E(X)$ for all $X\in \cX_+$, the same equation holds for all $X\in \cX$ by linearity of $\rho$. For a general $\rho$, invoke Lemma \ref{reduceto+} to consider $\rho^\pm$.
\end{proof}

\end{document}